\documentclass[preprint,12pt]{elsarticle}
\usepackage{amssymb}
\usepackage{lineno,hyperref}
\usepackage{amsmath}
\usepackage{amsthm}
\usepackage{graphicx}
\usepackage{mathrsfs}

\newtheorem{lemma}{Lemma}[section]
\newtheorem{theorem}{Theorem}[section]

\newtheorem{definition}{Definition}[section]

\journal{}

\begin{document}

\begin{frontmatter}



\title{Diffusion equations with general nonlocal time and space derivatives}


\author[Firstaddress]{Chung-Sik Sin \corref{mycorrespondingauthor}}
\cortext[mycorrespondingauthor]{Corresponding author}
\ead{cs.sin@ryongnamsan.edu.kp}

\address[Firstaddress]{Faculty of Mathematics, \textit{\textbf {Kim Il Sung}} University, Ryomyong District, Pyongyang, Democratic People's Republic of Korea}

\author[Firstaddress]{Hyong-Chol O}

\author[Firstaddress]{Sang-Mun Kim}

\begin{abstract}
In the present study, firstly, based on the continuous time random walk (CTRW) theory, general diffusion equations are derived.
The time derivative is taken as the general Caputo-type derivative introduced by Kochubei and
the spatial derivative is  the general Laplacian defined by removing the conditions (1.5) and  (1.6)  from the definition of the general fractional Laplacian proposed in the paper (Servadei and Valdinoci, 2012).
Secondly, the existence of solutions of the Cauchy problem for the general diffusion equation is proved by extending the domain of the general Laplacian to a general Sobolev space. The results for positivity and boundedness of the solutions are also obtained.
In the last, the existence result for solutions of the initial boundary value problem (IBVP) for the general diffusion equation on a bounded domain is established by using the Friedrichs extension of the general fractional Laplacian introduced in the book (Bisci, Radulescu and Servadei, 2016).
\end{abstract}

\begin{keyword}
general Caputo-type derivative, general  Laplacian, anomalous diffusion equation, continuous time random walk theory, general Sobolev space, existence of solution,
Friedrichs extension.
\end{keyword}

\end{frontmatter}


\section{Introduction}
\label{Section 1}
Anomalous diffusion phenomena have been observed in many physical processes such as diffusion processes inside living biological cells \cite{Barkai,Hofling,Sokolov}, motion of charge carrier in amorphous semiconductors \cite{Scher} and motion of suspended particle in ground water \cite{Benson}.
In particular, with the help of modern single particle superresolution techniques, motion of anomalously diffusing particle in complex crowded environment has been intensively investigated \cite {Norregaard,Eggeling,Manzo,Honigmann}. Recently, in \cite{Alexander}, anomalous diffusion processes were successfully simulated by Monte Carlo method.

It is a very interesting matter that the anomalous diffusion process can be described by means of a CTRW model.  
The probability density function (PDF)  of the CTRW process is derived from the waiting time PDF  and the jump length PDF \cite{Montroll}.
When the waiting time PDF $ \phi(t)$ of moving particle follows the relation: $ \phi(t)\approx t^{-1-\alpha}, t \rightarrow \infty $ for some $ \alpha \in (0,1) $, then the random walk process, whose mean squared displacement (MSD) is $  <x^2(t)> \approx t^\alpha, t \rightarrow \infty $, corresponds to the Caputo/Riemann-Liouville time fractional diffusion equation \cite{Metzler}. If the jump length PDF $ \lambda (x)$ satisfies the relation: $ \lambda (x)\approx |x|^{-1-\beta}, |x| \rightarrow \infty $ for some $ \beta \in (0,2) $, then  the L\'{e}vy process, whose MSD diverges, is modeled by the Riesz space fractional diffusion equation \cite{Metzler}.
However, classical fractional derivatives such as Caputo, Riemann-Liouville and Riesz fractional derivatives are insufficient to capture all of the anomalous diffusion phenomena. 
In fact, in \cite{Lomholt,Sanders}, the authors found the ultraslow diffusive processes, whose MSDs don't follow the power law in time but have logarithmic dependence.   
In \cite{Chechkin-Sokolov}, the authors used the fractional derivative of distributed order, which generalizes  single-term and multi-term fractional derivatives, to model the ultraslow diffusion process effectively. 

The diversity of the anomalous diffusion process needs a more general framework for modeling the complex diffusion phenomena.  Kochubei introduced a general derivative which generalizes the previous fractional derivatives such as  single-term, multi-term and distributed order Caputo-type derivatives \cite{Kochubei}.
It has been shown in the papers \cite{Sandev_FCAA,Sandev_Chaos} that the general derivative is very appropriate as time derivative in describing the anomalous diffusion process. On the other hand, in order to define a general space derivative,  Servadei and Valdinoci proposed a generalization of the fractional Laplacian and considered an equation involving the generalized operator \cite{Servadei_Mountain}. In \cite{Bisci}, the auhors investigated nonlocal elliptic equations involving the operator defined by removing a condition for the kernel function from the definition of the generalization of the fractional Laplacian introduced in \cite{Servadei_Mountain}.
However, the MSDs of the random walk processes described by the anomalous diffusion equations involving the generalizations of the fractional Laplacian are infinite. The most important problem which affects the applicability of L\'{e}vy statistics is the divergence of the MSD due to slowly  decaying tails \cite{del-Castillo}. In case of the exponentially tempered L\'{e}vy flights, which have been found in plasma physics and fluid mechanics, the divergence problem has been solved in \cite{del-Castillo}.

With the rapid development of mathematical modeling of anomalous diffusion processes, mathematical aspects for diffusion equations involving nonlocal derivatives have been widely investigated.
Eidelman and Kochubei studied the fundamental solution of the Cauchy problem of the Caputo-type time fractional diffusion equation by employing the H-functions  \cite{Eidelman}. 
In \cite{Luchko_2009}, a maximum principle of the IBVP of the single-term Caputo-type time fractional diffusion equation was proved.
In \cite{Sakamoto}, the authors used the eigenfunction expansion of the symmetric uniformly elliptic operator and the properties of the two-parameter Mittag-Leffler function to establish the existence of weak solutions of the IBVP of the single-term time fractional diffusion equation.
In \cite{GorLuc}, existence results for a suitably defined weak solution of the IBVP of the single-term time fractional diffusion equation were obtained in the fractional Sobolev space. 
In \cite{Luchko_multi_term,Li_Liu_Yamamoto}, existence and uniqueness of solutions of the multi-term Caputo-type time fractional diffusion equation were proved  by deriving the maximum principle for the IBVP and the new properties of the multivariate Mittag-Leffler function. 
In \cite{Li_2017,Rundell_Zhang}, the authors considered the distributed order time fractional diffusion equations.
In \cite{Mainardi}, the fundamental solution of the Cauchy problem for the fractional diffusion
equation with the Caputo fractional  derivative in  time and the Riesz-Feller derivative in  space was investigated.
In \cite{Luchko-Yamamoto}, existence and uniqueness of solutions of the IBVP for the fractional diffusion equation with the general Caputo-type derivative in time were established. 

In this paper, we try to develop a more general framework for describing more complex anomalous diffusion phenomena. 
By using the CTRW theory, we obtain a general diffusion equation involving the general Caputo-type differential operator and a nonlocal operator more general than the general fractional Laplacians proposed in \cite{Servadei_Mountain, Bisci}.
The advantage of the general diffusion equation  is that it can capture the L\'{e}vy process whose MSD is finite. 

The organization of the rest paper is as follows. In Section 2,  the general nonlocal derivatives are discussed  and the general nonlocal diffusion equations are derived from the CTRW theory.
In Section 3, we establish the existence, positivity and boundedness of solutions of the Cauchy problem for the general diffusion equation. For special cases of the general diffusion equation, the analytical solutions are represented in terms of Mittag-Leffler type functions.
In Section 4, we obtain the existence result for the IBVP of the general diffusion equation on a bounded domain.

\section{General diffusion equation derived from CTRW theory}
\label{Section 2}
\setcounter{section}{2}
\setcounter{equation}{0}\setcounter{theorem}{0}
In this section, we recall the concept of the general Caputo-type differential operator and propose the general Laplacian  including the generalizations of the fractional Laplacian introduced in \cite{Servadei_Mountain, Bisci}. Then, by means of the CTRW theory,  the general diffusion equation involving the general Caputo-type derivative in time and the general Laplacian in space is obtained.
\subsection{General Caputo-type derivative}
\begin{definition} [\cite{Kochubei}]
The general Caputo-type differential operator $ \mathscr{D}_{(g)} $ is defined by
\begin{equation}
\nonumber
\mathscr{D}_{(g)} v(t)=\frac{d}{dt}\int_0^t g(t-\tau)v(\tau)d\tau-g(t)v(0).
\end{equation}
\end{definition}
Throughout this paper, we assume that the kernel function $ g $ satisfies the following conditions:\\
(C1) The Laplace transform $ \Lambda g $ of $ g $
\begin{equation}
\nonumber
\Lambda g(s)=\hat g(s)=\int_0^\infty e^{-ts} g(t) dt
\end{equation}
 exists for all $ s>0 $, \\
(C2) $ \hat g(s) $ is a Stieltjes function,\\
(C3) $ \hat g(s) \rightarrow 0 $ and $ s \hat g(s) \rightarrow \infty $ as $ s \rightarrow \infty $,\\
(C4) $ \hat g(s) \rightarrow \infty $ and $ s \hat g(s) \rightarrow 0 $ as $ s \rightarrow 0 $.

Here the function $ v $ is said to be a Stieltjes function, if there exist $ a,b \geq 0 $ and a Borel measure $ \sigma $ on $ [0,\infty) $ such that 
\begin{equation}
v(t)=\frac{a}{t}+b+\int_0^{\infty} \frac{1}{s+t} \sigma (ds),
\end{equation}
where
\begin{equation}
\nonumber
\int_0^\infty (1+t)^{-1} \sigma (dt)<\infty.
\end{equation}

As a typical example of the general Caputo-type derivative, when $ g(t)=\frac{t^{-\alpha}}{\Gamma(1-\alpha)} $ for some $ 0<\alpha<1 $, we obtain $\mathscr{D}_{(g)} v(t)=D^\alpha v(t) $, where $ D^\alpha $ means the Caputo derivative.  For details of the Caputo derivative, see \cite{DieBoo}.

It was proved in \cite{Kochubei} that the kernel function $ g $ is completely monotone if $ g $ satisfies the conditions (C1)-(C4).
Here the function $ g $ is said to be a completely monotone function, if $ g\in C^\infty(0,\infty) $ and  $ (-1)^j g^{(j)}(t)\geq 0 $ for $ t>0, j=0,1,2,...$.
For details of completely monotone functions and Stieltjes functions, see \cite{Schilling}.

Meanwhile, by the results in \cite{Kochubei,Sin}, the following property is obtained.
\begin{lemma}
\label{property_of_general_Caputo}
 Let $ \lambda>0 $. Then the problem
\begin{align}
\label{homogenous_linear_differential_equation}
\mathscr{D}_{(g)}v(t)&=-\lambda v(t), t>0,\\
\label{1_initial_condition}
&v(0)=1
\end{align}
has a unique solution in $ C[0,\infty)  $. In particular, the solution is a completely monotone function on $ (0,\infty) $.
\end{lemma}

\subsection{General Laplacian}
\label{subsection_general_fractional_Laplacian}
\begin{definition}
\label{definition_general_laplacian}
Let the function $ k:\mathbb{R}^n \rightarrow (0,\infty) $ satisfy the condition: $ \min\{|x|^2,1\}k(x) \in L^1(\mathbb{R}^n) $.
Then the general  Laplacian $ \mathscr{L}_{(k)} $ is defined by
\begin{equation}
\label{general_laplacian}
\mathscr{L}_{(k)}v(x)=\int_{\mathbb{R}^n} \big(v(x+y)+v(x-y)-2v(x)\big)k(y)dy,  x\in \mathbb{R}^n.
\end{equation}
\end{definition}

For example, when $ k(x)=|x|^{-(n+2\beta)} $ for some $ 0<\beta< 1 $,  we obtain $\mathscr{L}_{(k)}=-c(-\triangle)^\beta$, where $(-\triangle)^\beta$ denotes the fractional Laplacian and $ c>0 $ means the normalization constant. The fractional Laplacian is discussed in detail in the reference \cite{Landko}. 

We remark that Definition \ref{definition_general_laplacian} is obtained by eliminating, from the concept of the generalization of the fractional Laplacian introduced in \cite[subsection 1.3.3]{Bisci},  the condition \\
(C5) There exist $ \theta>0, 0<\beta< 1 $ such that $ k(x)\geq \theta |x|^{-(n+2\beta)} $ for $ x \in \mathbb{R}^n \setminus \{0\} $.\\
Moreover, Definition \ref{definition_general_laplacian} is given by removing, from the concept of the general fractional Laplacian proposed by Servadei and Valdinoci in \cite{Servadei_Mountain},  the condition (C5) and the other condition: $k(x)=k(-x) \text{ for any } x\in \mathbb{R}^n \setminus \{0\} $. \\

An example of the general Laplacian  $ \mathscr{L}_{(k)} $ which doesn't satisfy the condition (C5) is one whose kernel function is of the exponentially truncated form
\begin{equation}
\label{typical_general_exam}
k(x)=q|x|^{-(1+2\beta)}e^{-h|x|}  \text{ for } x \in \mathbb{R},  q>0,  0<\beta < 1 \text{ and } h > 0. 
\end{equation}

In the following lemma, an alternative definition of the general Laplacian  $ \mathscr{L}_{(k)} $ is introduced
by using the Fourier transform. In other words, the general Laplacian  $ \mathscr{L}_{(k)} $ can be written as a pseudodifferential operator.
\begin{lemma}
\label{property_of_general_Lapalcian}
Let the function $ k:\mathbb{R}^n \rightarrow (0,\infty) $ satisfy the condition:
\begin{equation}
\nonumber
\min\{|x|^2,1\}k(x) \in L^1(\mathbb{R}^n) 
\end{equation}
 and the function $ \zeta: \mathbb{R}^n \setminus \{0\} \rightarrow (0, \infty) $ be defined by
\begin{equation}
\label{zeta}
\zeta(x)=2\int_{\mathbb{R}^n} (1-\cos(xy))k(y)dy, x\in \mathbb{R}^n.
\end{equation}
Then the following relation holds.
\begin{equation}
\nonumber
\mathscr{L}_{(k)}v(x)=-F^{-1}(\zeta(\xi)Fv(\xi))(x), v\in S(\mathbb{R}^n), x\in \mathbb{R}^n,
\end{equation}
where $ F, F^{-1} $ are respectively Fourier transform and inverse Fourier transform defined by
\begin{align}
\nonumber
&Fv(\xi)=\tilde{v}(\xi)=\frac{1}{(2\pi)^\frac{n}{2}}\int_{\mathbb{R}^n} v(x) e^{-ix\xi} dx,\\
\nonumber
&F^{-1}v(x)=\frac{1}{(2\pi)^\frac{n}{2}} \int_{\mathbb{R}^n} v(\xi) e^{ix\xi} d\xi.
\end{align}
Here $ S(\mathbb{R}^n) $ means the Schwartz space of rapidly decreasing smooth functions $ v $ satisfying 
\begin{equation}
\nonumber
 \sup_{x=(x_1,...x_n) \in \mathbb{R}^n} \bigg|x_1^{i_1}\cdots x_n^{i_n}\frac{\partial^{j_1}v(x)}{\partial x_1^{j_1}}\cdots\frac{\partial^{j_n}v(x)}{\partial x_n^{j_n}}\bigg|<+\infty 
\end{equation} 
for any $ i_1,...,i_n, j_1,...,j_n \in \mathbb{N} \cup \{0\}$.
\end{lemma}
\begin{proof}
For $ v\in S(\mathbb{R}^n) $, we have
\begin{equation}
\nonumber
\big(v(x+y)+v(x-y)-2v(x)\big)k(y)\leq \|D^2v\|_{L^{\infty}(\mathbb{R}^n)}|y|^2k(y).
\end{equation}
Thus $\mathscr{L}_{(k)}v$ exists.
For $ v\in S(\mathbb{R}^n) $ and $ \xi \in \mathbb{R}^n $, we have 
\begin{align}
\nonumber
F\{\mathscr{L}_{(k)}v\}(\xi)&=\int_{\mathbb{R}^n} k(y) F\big(v(x+y)+v(x-y)-2v(x)\big)(\xi)dy\\
\nonumber
&=\int_{\mathbb{R}^n} k(y) \big(e^{i\xi y}+e^{-i\xi y}-2\big)dy Fv(\xi)\\
\nonumber
&=-2\int_{\mathbb{R}^n} k(y) \big(1-\cos(\xi y)\big)dy Fv(\xi)\\
\nonumber
&=-\zeta(\xi)Fv(\xi).
\end{align}
\end{proof}
Throughout this paper, we assume that the kernel function $ k $ of the general Laplacian $\mathscr{L}_{(k)}$ satisfies the following condition.\\
(C6) $ \zeta $ defined by (\ref{zeta}) is negative definite.\\
Here we say that the function $ v: \mathbb{R}^n \rightarrow \mathbb{C} $ is negative definite, if $ v(x)=v(-x) $ for $ x \in \mathbb{R}^n $ and the relation
\begin{equation}
\nonumber
\sum_{j,k=1}^m (v(x_j)+\overline{v(x_k)}-v(x_j-x_k))c_j\overline{c}_k\geq 0
\end{equation} 
holds for $ m\in \mathbb{N}, x_1,...,x_m \in \mathbb{R}^n $ and $ c_1,...,c_m \in \mathbb{C} $. 
If  the relation
\begin{equation}
\nonumber
\sum_{j,k=1}^m v(x_j-x_k)c_j\overline{c}_k\geq 0
\end{equation} 
holds for $ m\in \mathbb{N}, x_1,...,x_m \in \mathbb{R}^n $ and $ c_1,...,c_m \in \mathbb{C} $,
we say that the function $ v:\mathbb{R}^n \rightarrow \mathbb{C}$ is positive definite.
It follows from Theorem 13.14 in \cite{Schilling} that the following condition is a sufficent condition for $ \zeta $ to be negative definite:\\
there exists a Bernstein function $  f:(0, \infty) \rightarrow [0, \infty) $ such that for  $ x \in \mathbb{R}^n $, the relation $ f(|x|^2)=\zeta(x) $ holds. \\
L\'{e}vy-Khintchine's Theorem \cite[Theorem 4.15]{Schilling} gives a representation for the continuous negative definite function.
We can see that the examples of kernels of $\mathscr{L}_{(k)}$ mentioned in the above satisfy the condition (C6).

\subsection{General diffusion equation}
The probability density function (PDF) of an uncoupled CTRW process in Fourier-Laplace space is written by \cite{Metzler}
\begin{equation}
\label{pdf_fourier_laplace_space}
\hat{\tilde{p}}(s,\xi)=\frac{1-\hat{\phi}(s)}{s}\frac{1}{1-\hat{\phi}(s)\tilde{w}(\xi)},
\end{equation}
where $\hat{\tilde{p}}(s,\xi)$ stands for the Fourier-Laplace transform of the PDF $ p(t,x) $ to find a walker at position $ x $ at time $ t $, $\hat{\phi}(s) $  means the Laplace transform of the waiting time PDF $ \phi(t) $ and
$ \tilde{w}(\xi) $ denotes the Fourier transform of the jump length PDF $ w(x) $.

Set
\begin{align}
\label{time_pdf}
\hat{\phi}(s)=\frac{1}{1+s\hat{g}(s)},\\
\label{length_pdf}
 \tilde{w}(\xi)=e^{-\zeta(\xi)},
\end{align}
where $ g:\mathbb{R} \rightarrow \mathbb{R} $ is a function satisfying the conditions (C1)-(C4) and $ \zeta:\mathbb{R}^n \rightarrow \mathbb{R}  $ is a function satisfying the relation (C6).
In order to guarantee  that a function is a PDF, its Laplace transform should be completely monotone \cite{Feller}.
Since $ g $ satisfies the conditions (C1) and (C2), using the properties of the Stieltjes function, we can easily see that 
$\hat{\phi}(s)$ is a Stieltjes function.
Meanwhile, if  $\phi(t)$ is a PDF, then the following relations 
\begin{equation}
\label{requirement_2}
1=\int_0^\infty \phi(t) dt=\lim_{s \to 0} \hat{\phi}(s)
\end{equation}
and
\begin{equation}
\label{requirement_3}
0=\lim_{s \to \infty} \int_0^\infty e^{-st} \phi(t) dt=\lim_{s \to \infty}\hat{\phi}(s)
\end{equation}
hold.
Since $ g $ satisfies the conditions (C3) and (C4), the relations  (\ref{requirement_2}) and (\ref{requirement_3}) are satisfied.
In order to guarantee that a function is a PDF, its Fourier transform should be positive definite \cite{Pinsky}. By the condition (C6) and Proposition 4.4 in \cite{Schilling},  the function $ e^{-\zeta(\xi)} $ is positive definite.  Meanwhile, the other condition for $ w(x) $ to be a PDF is $ \tilde{w}(0)=1  $, which is satisfied by the formulas (\ref{zeta}) and (\ref{length_pdf}).
By the formulas  (\ref{pdf_fourier_laplace_space}), (\ref{time_pdf}) and the estimation
\begin{equation}
\nonumber
e^{-\zeta(\xi)} \approx 1-\zeta(\xi), \xi \to 0, 
\end{equation}
we have
\begin{equation}
\label{pdf_time_space}
\hat{\tilde{p}}(s,\xi)=\frac{\hat{g}(s)}{s\hat{g}(s)+\zeta(\xi)}.
\end{equation}
We can rewrite (\ref{pdf_time_space}) as
\begin{equation}
\nonumber
\hat{g}(s)s\hat{\tilde{p}}(s,\xi)-\hat{g}(s)=-\zeta(\xi)\hat{\tilde{p}}(s,\xi).
\end{equation}
Using the inverse Fourier-Laplace transform, we can easily obtain the general anomalous diffusion equation
\begin{equation}
\label{diffusion_equation}
\mathscr{D}_{(g)} p(t,x)=\mathscr{L}_{(k)} p(t,x), t>0, x\in \mathbb{R}^n.
\end{equation}

When $ g(t)=\frac{t^{-\alpha}}{\Gamma(1-\alpha)} $ for some $ 0<\alpha < 1 $ and $ k(x)=c^{-1}|x|^{-(n+2\beta)} $ for some $ 0<\beta< 1 $, the equation (\ref{diffusion_equation}) corresponds to the Caputo-Riesz time-space fractional diffusion equation
\begin{equation}
\label{Caputo_Riesz_diffusion_equation}
D^{\alpha} p(t,x)=-(-\triangle)^\beta p(t,x), t>0, x\in \mathbb{R}^n.
\end{equation}

If $ g(t)=\sum_{j=1}^m a_j \frac{t^{-\alpha_j}}{\Gamma(1-\alpha_j)} $ for some $ a_1,...,a_m>0, 
0<\alpha_m<\cdots<\alpha_1< 1 $ and $ k(x)=\sum_{j=1}^r \frac{b_j}{c}|x|^{-(n+2\beta_j)} $ for some $ b_1,...,b_r>0, 0<\beta_r<\cdots<\beta_1< 1 $, the equation (\ref{diffusion_equation}) yields the multi-term fractional diffusion equation
\begin{equation}
\label{multi_Caputo_Riesz_diffusion_equation}
\sum_{j=1}^m a_jD^{\alpha_j} p(t,x)=-\sum_{j=1}^rb_j(-\triangle)^{\beta_j} p(t,x), t>0, x\in \mathbb{R}^n.
\end{equation}

Setting $ g(t)=e^{-bt} \frac{t^{-\alpha}}{\Gamma(1-\alpha)} $ for some $ b>0, 0<\alpha< 1 $ and $ k(x)=c^{-1}|x|^{-(n+2\beta)} $ for some $ 0<\beta< 1 $,
the equation (\ref{diffusion_equation}) is of the form
\begin{align}
\nonumber
&\frac{1}{\Gamma(1-\alpha)}\frac{\partial}{\partial t}\int_0^t e^{-b(t-\tau)}(t-\tau)^{-\alpha} p(\tau,x)d\tau-e^{-bt} \frac{t^{-\alpha}}{\Gamma(1-\alpha)}p(0,x)\\
\label{tempered_Riesz_diffusion_equation}
&=-(-\triangle)^\beta p(t,x), t>0, x\in \mathbb{R}^n.
\end{align}
Here we note that for $ b=0 $, the equation (\ref{tempered_Riesz_diffusion_equation}) corresponds to  the equation (\ref{Caputo_Riesz_diffusion_equation}). 
We remark that when the condition: $\hat g(s) \rightarrow \infty $ as $ s \rightarrow 0 $ in the assumption (C4) for $g$ doesn't hold, we can still prove the same results as that of the present paper.  

The above mentioned examples are all ones whose MSDs diverge.
However, from the physical view point, displacements of randomly walking particles can never be arbitrarily large.
In equations (\ref{Caputo_Riesz_diffusion_equation}) and (\ref{tempered_Riesz_diffusion_equation}), if $ \beta=1 $, then we get the time fractional diffusion equations considered in detail in the previous publication \cite{Sandev_FCAA}. In this case, the MSD of the corresponding random walk process is finite. However, in general, we can never say that the jump length PDFs of all CTRW processes are of the standard form.
Thus, it is necessary to develop mathematical models of anomalous processes that incorporate large displacement events while keeping the second moments finite.

We can easily see that if $ \max\{x^2,1\}k(x) \in L^1(\mathbb{R}) $, then, using the relations $ \zeta(0)=0 $ and $ \zeta'(0)=0 $, the MSD is finite and
\begin{equation}
\nonumber
\hat{M_2}(s)=-\frac{\partial^2\hat{\tilde{p}}(s,\xi)(s,\xi)}{\partial \xi^2}\bigg|_{\xi=0}=\frac{\zeta''(0)}{s^2\hat{g}(s)}=
\frac{2q}{s^2\hat{g}(s)}\int_{\mathbb{R}} y^2k(y)dy.
\end{equation}
When $ k(x) $ is of the form (\ref{typical_general_exam}), we have
\begin{equation}
\label{zeta_expression}
\zeta(\xi)=2\int_{\mathbb{R}} (1-\cos(\xi y))q|y|^{-(1+2\beta)}e^{-h|y|}dy.
\end{equation}
Then, in one dimensional space, the Laplace transform of the mean squared displacement $ M_2(t) $ is of the form
\begin{equation}
\nonumber
\hat{M_2}(s)=\frac{\zeta''(0)}{s^2\hat{g}(s)}=\frac{2q}{s^2\hat{g}(s)}\int_{\mathbb{R}} |y|^{1-2\beta}e^{-h|y|}dy.
\end{equation}

If $ g(t)=e^{-bt} \frac{t^{-\alpha}}{\Gamma(1-\alpha)} $ for some $ b>0, 0<\alpha< 1 $  and 
 $ k(x) $ is taken as the form (\ref{typical_general_exam}), we obtain the following diffusion equation
\begin{align}
\nonumber
&\frac{1}{\Gamma(1-\alpha)}\frac{\partial}{\partial t}\int_0^t e^{-b(t-\tau)}(t-\tau)^{-\alpha} p(\tau,x)d\tau-e^{-bt} \frac{t^{-\alpha}}{\Gamma(1-\alpha)}p(0,x)\\
\label{tempered_time_space_diffusion_equation}
&=-F^{-1}(\zeta(\xi)Fp)(t,x) , t>0, x\in \mathbb{R},
\end{align}
where $ \zeta(\xi) $ is of the form (\ref{zeta_expression}).

\section{Cauchy problem of general diffusion equation}
\label{Section 3}
\setcounter{section}{3}
\setcounter{equation}{0}\setcounter{theorem}{0}
In this section, we investigate the general diffusion equation (\ref{diffusion_equation}) subject to the initial condition
\begin{equation}
\label{initial_condition}
p(0,x)=f(x), x \in \mathbb{R}^n.
\end{equation}

\subsection{Extension of domain of general Laplacian}

Based on Lemma \ref{property_of_general_Lapalcian}, we extend the domain of the general Laplacian $ \mathscr{L}_{(k)} $  to a Banach space.
\begin{definition}
\label{general_fractional_Sobolev_space_definition}
We define the general Sobolev space $ M_{(k)} $ by
\begin{equation}
\nonumber
M_{(k)}=\{ u\in L^2(\mathbb{R}^n):\zeta(\xi)Fu(\xi)\in L^2(\mathbb{R}^n)\},
\end{equation}
where $ \zeta $ is defined by formula (\ref{zeta}).
\end{definition}

\begin{theorem}
\label{banach_space}
The space $ M_{(k)} $ with the norm $ \|\cdot\|_{M_{(k)}} $ defined by
\begin{equation}
\nonumber
\|u\|_{M_{(k)}}=\|(1+\zeta(\xi))Fu(\xi)\|_{L^2(\mathbb{R}^n)}
\end{equation}
is a Banach space.
\end{theorem}
\begin{proof}
Let $\{ u_j \}$ be a Cauchy sequence in $ M_{(k)} $.
Then there exists a $ v \in L^2(\mathbb{R}^n) $ such that
\begin{equation}
\nonumber
\lim_{j \rightarrow \infty}\|(1+\zeta(\xi))Fu_j(\xi)\|_{ L^2(\mathbb{R}^n)}=\|v\|_{ L^2(\mathbb{R}^n)}.
\end{equation}
Let $ f(\xi)=(1+\zeta(\xi))^{-1} $.
Then it is obvious that $ fv\in L^2(\mathbb{R}^n) $.
Setting $ u=F^{-1}(fv) $, we can easily see that $ u\in M_{(k)} $.
It follows from the continuity of Fourier transform in $ L^2(\mathbb{R}^n) $ that $ \lim\limits_{j \rightarrow \infty} \|u_j-u\|_{M_{(k)}}=0 $.
\end{proof}

\begin{theorem}
\label{density}
The Schwartz space $ S(\mathbb{R}^n) $ is dense in $ M_{(k)} $.
\end{theorem}
\begin{proof}
Let $ u \in  M_{(k)}  $.
It follows from the density of the space $ C_0^{\infty}(\mathbb{R}^n) $ in $L^2(\mathbb{R}^n)$ that there exists a sequence 
$\{v_j\}\in C_0^{\infty}(\mathbb{R}^n) $ such that 
\begin{equation}
\nonumber
\lim_{j \rightarrow \infty}\|v_j(\xi)-(1+\zeta(\xi))Fu(\xi)\|_{ L^2(\mathbb{R}^n)}=0.
\end{equation}
Then the function $ u_j(\xi)= (1+\zeta(\xi))^{-1}v_j(\xi)$ belongs to $ C_0^{\infty}(\mathbb{R}^n) $ and converges to $Fu$ in  $L^2(\mathbb{R}^n)$.
The continuity of $ F^{-1} $ implies that $ F^{-1}u_j $ converges to $ u $. 
It is obvious the $ F^{-1}u_j $ is in $ S(\mathbb{R}^n) $.
\end{proof}
\begin{theorem}
The general Laplacican $ \mathscr{L}_{(k)} $ is extended to the Banach space $ M_{(k)} $.
\end{theorem}
\begin{proof}
By the density of $ S(\mathbb{R}^n) $ in $ M_{(k)} $ and the extension principle, the desired result is easily proved.
\end{proof}

Here we note that when $ k(x)=|x|^{-(n+2\beta)} $ for some $ 0<\beta<1 $, $ M_{(k)} $ corresponds to the fractional Sobolev space introduced in detai in \cite{Nezza}.

\subsection{Existence  and properties of solutions}

\begin{lemma}
\label{property of CMF}
Let $ g $ be a completely monotone function.  
Then the following relations hold.\\
(a) If $ g(0+)< \infty $, 
\begin{equation}
\nonumber
|g^{(m)}(t)|< {g(0+)}\bigg(\frac{m}{et}\bigg)^m, m \in \mathbb{N}, t>0.
\end{equation}
(b) If $ \lim\limits_{t \to \infty} g(t)=0 $,
\begin{equation}
\nonumber
\lim_{t \to \infty} \frac{\|g\|_{L^1(0,t)}}{t}=0.
\end{equation}
\end{lemma}

\begin{proof}
Firstly, we prove the part (a) of the lemma.
Since  the function $ g $ is completely monotone, by Bernstein's theorem \cite{Pruss}, there exists a nondecreasing function $ \sigma:[0, \infty) \rightarrow  \mathbb{R} $ such that 
\begin{equation}
\nonumber
g(t)=\int_0^\infty e^{-ts} d\sigma(s), t>0,
\end{equation}
where $ \sigma(0)=0 $ and $  \sigma(t) $ is continuous from the left.
Moreover,
\begin{equation}
\nonumber
(-1)^mg^{(m)}(t)=\int_0^\infty s^me^{-ts}d\sigma (s), m=0,1,2,..., t>0,
\end{equation}
and
\begin{equation}
\nonumber
(-1)^mg^{(m)}(0+)=\int_0^\infty s^md\sigma (s).
\end{equation}
Then, for $m=0$, we have
\begin{equation}
\nonumber
g(0+)=\int_0^\infty d\sigma (s)=\sigma(\infty).
\end{equation}
We can easily see that
\begin{equation}
\nonumber
 \sup_{s \in [0,\infty)}s^me^{-ts}=\bigg(\frac{m}{et}\bigg)^m,  m\in \mathbb{N}, t>0.
 \end{equation}
Thus, we have
\begin{equation}
\nonumber
|g^{(m)}(t)|=\int_0^\infty s^me^{-ts}d\sigma (s) < g(0+)\bigg(\frac{m}{et}\bigg)^m, m \in \mathbb{N}, t>0.
\end{equation}
Secondly, we  prove the part (b) of the lemma.
By Fubini's theorem, for $ s>0 $, we have
\begin{align}
\nonumber
&g(s)=\int_0^\infty e^{-sw} d\sigma(w)= \int_0^\infty \int_w^\infty se^{-s \tau} d\tau d\sigma(w)=\int_0^\infty \int_0^\tau se^{-s \tau} d\sigma(w) d\tau\\
\nonumber
&=\int_0^\infty se^{-s \tau} \sigma(\tau) d\tau.
\end{align}
For $ t>0 $, we deduce
\begin{align}
\nonumber
&\|g\|_{L^1(0,t)}=\int_0^t \int_0^\infty se^{-s \tau} \sigma(\tau) d\tau ds=\int_0^\infty \int_0^tse^{-s \tau} ds \sigma(\tau) d\tau\\
\nonumber
&=\int_0^\infty \frac{1}{\tau^2}(1-e^{-t\tau}-t\tau e^{-t\tau}) \sigma(\tau) d\tau.
\end{align}
Then, for $ t>0 $,
\begin{align}
\nonumber
&\frac{\|g\|_{L^1(0,t)}}{t}= \int_0^\infty \frac{1-e^{-t\tau}-t\tau e^{-t\tau}}{t\tau^2} \sigma(\tau) d\tau 
\leq  \int_0^\infty \frac{1-e^{-t\tau}-t\tau e^{-t\tau}}{t\tau^2}  \sigma(\infty) d\tau\\
\nonumber
&= \int_0^\infty \frac{1-e^{-w}-w e^{-w}}{w^2}  \sigma(\infty) dw.
\end{align}
Since 
\begin{equation}
\nonumber
\lim_{w \to 0} \frac{1-e^{-w}-w e^{-w}}{w^2}=\frac{1}{2}
\end{equation}
and
\begin{equation}
\nonumber
\sup_{w\in [0,\infty)}\{1-e^{-w}-w e^{-w}\}=1,
\end{equation}
we have
\begin{equation}
\nonumber
\int_0^\infty \frac{1-e^{-w}-w e^{-w}}{w^2} dw=\int_0^\delta  dw+\int_\delta^\infty \frac{1}{w^2} dw<\infty,
\end{equation}
where $ \delta>0 $ is a real number such that
\begin{equation}
\nonumber
 \frac{1-e^{-w}-w e^{-w}}{w^2} < 1, w \in (0, \delta).
\end{equation}
It follows from Lebesgue's dominated convergence theorem that
\begin{align}
\nonumber
&\lim_{t \to \infty}\frac{\|g\|_{L^1(0,t)}}{t}=\lim_{t \to \infty}  \int_0^\infty \frac{1-e^{-t\tau}-t\tau e^{-t\tau}}{t\tau^2} \sigma(\tau) d\tau\\
\nonumber
&= \int_0^\infty \lim_{t \to \infty} \frac{1-e^{-t\tau}-t\tau e^{-t\tau}}{t\tau^2} \sigma(\tau) d\tau=0.
\end{align}
\end{proof}

\begin{theorem}
Let $ f \in M_{(k)} $.
Then the general diffusion equation (\ref{diffusion_equation}) with the initial condition (\ref{initial_condition}) has a unique solution
$p\in  C([0,\infty), M_{(k)}) $.
The following relations  hold.
\begin{align}
\label{p_L2}
&\|p(t,\cdot)\|_{L^2(\mathbb{R}^n)}\leq \|f\|_{L^2(\mathbb{R}^n)}, t \geq 0.\\
\label{p_Mk}
&\|p(t,\cdot)\|_{M_{(k)}}\leq \|f\|_{M_{(k)}}, t \geq 0.\\
\label{p_zero_Mk}
&\lim_{t \to 0}\|p(t,\cdot)-f\|_{M_{(k)}}=0.\\
\label{p_derivative_L2}
&\bigg\|\frac{\partial^m p(t,\cdot)}{\partial t^m}\bigg\|_{L^2(\mathbb{R}^n)}\leq \bigg(\frac{m}{et}\bigg)^m\|f\|_{L^2(\mathbb{R}^n)}, m \in \mathbb{N}, t>0 .\\
\label{p_derivative_Mk}
&\bigg\|\frac{\partial^m p(t,\cdot)}{\partial t^m}\bigg\|_{M_{(k)}}\leq \bigg(\frac{m}{et}\bigg)^m\|f\|_{M_{(k)}},m \in \mathbb{N}, t>0 .\\
\label{Dg}
&\|D_{(g)}p(t,\cdot)\|_{L^2(\mathbb{R}^n)}\leq\|f\|_{M_{(k)}}, t>0.\\
\label{p_infty_L2}
&\lim_{t \to \infty}\|p(t,\cdot)\|_{M_{(k)}}=0.
\end{align}
If $  f  $ is nonnegative, then the solution is also nonnegative. 
In particular, if $  f  $ is bounded, then the solution is also  bounded.
\end{theorem}

\begin{proof}
Taking the Fourier transform to (\ref{diffusion_equation}) and (\ref{initial_condition}) with respect to the spatial variable $ x $, we obtain
\begin{equation}
\label{general_fractional_Laplace_diffusion_Fourier}
\mathscr{D}_{(g)} \tilde{p}(t,\xi) = -\zeta{(\xi)} \tilde{p}(t,\xi), t>0, \xi \in \mathbb{R}^n
\end{equation}
and
\begin{equation}
\label{initial_condition_Fourier}
\tilde{p}(0,\xi)=\tilde{f}(\xi).
\end{equation}
By  Theorem 3.2 in \cite{Sin}, the initial value problem (\ref{general_fractional_Laplace_diffusion_Fourier})-(\ref{initial_condition_Fourier}) has a unique solution in the space of continuous functions. In particular, the unique solution is of the  form: $\tilde{p}(t,\xi)=\tilde{f}(\xi)Z(t,\xi)$.
Here $Z(t,\xi)$ is the solution of the equation (\ref{general_fractional_Laplace_diffusion_Fourier}) with the initial condition $\tilde{p}(0,\xi)=1$. By Lemma \ref{property_of_general_Caputo}, $Z(t,\xi)$ is completely monotone with respect to $ t $. Then we can conclude
\begin{equation}
\nonumber
|\tilde{p}(t,\xi)| \leq |\tilde{f}(\xi)|, t>0, \xi \in \mathbb{R}^n.
\end{equation}
By Theorem 5.2 in \cite{Sin}, $ Z(t,\xi)$ continuously depends on $\xi$ . 
Thus, $ Z(t,\xi)$ is measurable with respect to $ \xi $.  Since $ f\in L^2(\mathbb{R}^n) $,  $  \tilde{f} \in L^2(\mathbb{R}^n) $. Then the function $ p(t,\cdot) $ is also in $ L^2(\mathbb{R}^n)$. Using the inverse Fourier transform, $ p(t,x) $ is obtained from $ \tilde{p}(t,\xi) $.
If $ f \in M_{(k)} $, then $ p(t,\cdot) $ is also in $ M_{(k)} $.
We can easily see that $ p(t,x) $ is continuous with respect to $ t $ by using the continuity of $ Z(t,\xi)  $ with respect to $ t $.

By Plancherel's theorem, for $ t\geq 0 $, we have
\begin{align}
\nonumber
&\|p(t,\cdot)\|_{L^2(\mathbb{R}^n)}=\|\tilde{p}(t,\cdot)\|_{L^2(\mathbb{R}^n)}=\|\tilde{f}(\cdot)Z(t,\cdot)\|_{L^2(\mathbb{R}^n)}\leq \|\tilde{f}\|_{L^2(\mathbb{R}^n)}=\|f\|_{L^2(\mathbb{R}^n)},\\
\nonumber
&\|p(t,\cdot)\|_{M_{(k)}}=\|(1+\zeta(\cdot))\tilde{p}(t,\cdot)\|_{L^2(\mathbb{R}^n)}\leq \|(1+\zeta(\cdot))\tilde{f}(\cdot)\|_{L^2(\mathbb{R}^n)} = \|f\|_{M_{(k)}}.\end{align}
For $ t>0 $,  we deduce
\begin{align}
\nonumber
&\|p(t,\cdot)-f\|_{M_{(k)}}=\|(1+\zeta(\cdot))(\tilde{p}(t,\cdot)-\tilde{f})\|_{L^2(\mathbb{R}^n)}\\
\nonumber
&=\|(1+\zeta(\cdot))\tilde{f}(\cdot) (1-Z(t,\cdot))\|_{L^2(\mathbb{R}^n)}\leq 2\|f\|_{M_{(k)}}.
\end{align}
Then it follows from  Lebesgue's dominated convergence theorem that
\begin{align}
\nonumber
&\lim_{t \to 0}\|p(t,\cdot)-f\|_{M_{(k)}}=\lim_{t \to 0} \|(1+\zeta(\cdot))(\tilde{p}(t,\cdot)-\tilde{f})\|_{L^2(\mathbb{R}^n)}\\
\nonumber
&=\|(1+\zeta(\cdot))\tilde{f}(\cdot)\lim_{t \to 0} (1-Z(t,\cdot))\|_{L^2(\mathbb{R}^n)}=0.
\end{align}
By (a) of Lemma \ref{property of CMF}, for $m \in \mathbb{N}, t>0 $,  we estimate
\begin{align}
\nonumber
&\bigg\|\frac{\partial^m \tilde{p}(t,\cdot)}{\partial t^m}\bigg\|_{L^2(\mathbb{R}^n)}=\bigg\|\tilde{f}(\cdot)\frac{\partial^m Z(t,\cdot)}{\partial t^m}\bigg\|_{L^2(\mathbb{R}^n)}\leq \bigg(\frac{m}{et}\bigg)^m\|\tilde{f}\|_{L^2(\mathbb{R}^n)}\\
\nonumber
&=\bigg(\frac{m}{et}\bigg)^m \|f\|_{L^2(\mathbb{R}^n)}.
\end{align}
Similarly, we can prove the relation (\ref{p_derivative_Mk}).
For $ t>0 $, we have
\begin{align}
\nonumber
&\|D_{(g)}p(t,\cdot)\|_{L^2(\mathbb{R}^n)}=\|D_{(g)}\tilde{p}(t,\cdot)\|_{L^2(\mathbb{R}^n)}=\|\zeta(\cdot)\tilde{p}(t,\cdot)\|_{L^2(\mathbb{R}^n)}\\
\nonumber
&=\|\zeta(\cdot)\tilde{f}(\cdot)Z(t,\cdot)\|_{L^2(\mathbb{R}^n)}\leq\|\zeta(\cdot)\tilde{f}(\cdot)\|_{L^2(\mathbb{R}^n)}\leq \|f\|_{M_{(k)}}.
\end{align}

Now we consider the asymptotic behaviour of solutions for the large times.
By Young's inequality, for $ \xi \in \mathbb{R}^n, t>0 $, we estimate
\begin{align}
\nonumber
&\zeta{(\xi)} t Z(t,\xi) \leq \|\zeta{(\xi)} Z(\cdot,\xi) \|_{L^1(0,t)}=\|\mathscr{D}_{(g)} Z(\cdot,\xi) \|_{L^1(0,t)}\\
\nonumber
&\leq \|g\|_{L^1(0,t)}\int_0^t \bigg| \frac{\partial Z(s,\xi)}{\partial s}\bigg| ds= \|g\|_{L^1(0,t)}(1- Z(t,\xi)).
\end{align}
Then 
\begin{equation}
\nonumber
Z(t,\xi)\leq \frac{1}{1+\frac{t\zeta(\xi)}{\|g\|_{L^1(0,t)}}}, t>0 , \xi \in \mathbb{R}^n.
\end{equation}
It follows from (b) of Lemma \ref{property of CMF} and Lebesgue's dominated convergence theorem that 
\begin{align}
\nonumber
&\lim_{t \to \infty}\|p(t,\cdot)\|_{M_{(k)}}=\lim_{t \to \infty} \|(1+\zeta(\cdot)) \tilde{p}(t,\cdot)\|_{L^2(\mathbb{R}^n)}\\
\nonumber
&=\|(1+\zeta(\cdot))\tilde{f}(\cdot) \lim_{t \to \infty}Z(t,\cdot)\|_{L^2(\mathbb{R}^n)}=0.
\end{align}

Since the function $ \hat{g}(s) $ is a Stieltjes function, by Theorem 6.2  in \cite{Schilling}, the function $ s\hat{g}(s) $ is a complete Bernstein function.
It follows from Theorem 7.5  in \cite{Schilling} that the function $ e^{-s\hat{g}(s)} $ is a Stieltjes function.
By the conditions (C3) and (C4), we deduce
\begin{equation}
\nonumber
\lim_{s \to \infty} e^{-s\hat{g}(s)}=0, \lim_{s \to 0} se^{-s\hat{g}(s)}=0. 
\end{equation}
Thus, the function $ e^{-s\hat{g}(s)} $ is a Laplace transform of a completely monotone function. 
Then 
\begin{equation}
\nonumber
\Lambda^{-1}\{\hat{g}(s)e^{-s\hat{g}(s)} \}(t)=\int_0^t g(t-\tau)\Lambda^{-1}\{e^{-s\hat{g}(s)} \}(\tau) d\tau.
\end{equation}
Define the function $ \phi(t,\tau) $ by 
\begin{equation}
\nonumber
\phi(t,\tau)=\Lambda^{-1}\{\hat{g}(s)e^{-\tau s\hat{g}(s)} \}(t), t,\tau>0.
\end{equation}
Since the set of completely monotone functions is closed under multiplication,  the function $ \hat{g}(s)e^{-s\hat{g}(s)} $ is also completely monotone. 
Then it follows from Bernstein's theorem that for $ t,\tau>0,$  $ \phi(t,\tau)>0. $
Using the complex inversion formula, for $ t>0 $, we have
\begin{equation}
\nonumber
\int_0^\infty \phi(t,\tau) d\tau= \int_0^\infty \frac{1}{2\pi i} \int_{r-i\infty}^{r+i\infty} \hat{g}(s)e^{st-\tau s\hat{g}(s)} ds d\tau
=\frac{1}{2\pi i} \int_{r-i\infty}^{r+i\infty} \frac{e^{st}}{s} ds=1,
\end{equation}
where $ r>0 $.
Meanwhile, we deduce
\begin{align}
\nonumber
&\int_0^\infty e^{-st} \int_0^\infty  \phi(t,\tau) e^{-\zeta(\xi) \tau}d\tau dt
=\hat{g}(s)\int_0^\infty   e^{-\tau s\hat{g}(s)} e^{-\zeta(\xi) \tau}d\tau \\
\nonumber
&=\frac{\hat{g}(s)}{s\hat{g}(s)+\zeta(\xi)}=\hat{ Z}(s,\xi), s>0, \xi \in \mathbb{R}^n.
\end{align}
By the uniqueness of the Laplace transform, $  Z(t,\xi) $ has the form:
\begin{equation}
\nonumber
Z(t,\xi)=\int_0^{\infty}  \phi(t,\tau) e^{-\zeta(\xi)\tau}  d\tau, t>0, \xi \in \mathbb{R}^n.
\end{equation}
Since $ \zeta(\xi) $ is negative definite, it follows from Proposition 4.4 in \cite{Schilling} that for $ \tau>0$, $e^{-\tau \zeta(\xi)}$ is positive definite with respect to $ \xi $. 
Thus,  $Z(t,\xi)$  is positive definite with respect to $ \xi $.
Then, by Bochner's theorem \cite[Theorem 3.9.16]{Pinsky}, there exists a finite nonnegative measure $  \mu_t $ on $ \mathbb{R}^n $ such that
\begin{equation}
\nonumber
Z(t,\xi)=\int_{\mathbb{R}^n} e^{-i \xi x}\mu_t (dx)=\tilde{\mu_t}(\xi), t>0, \xi \in \mathbb{R}^n.
\end{equation}
Then
\begin{equation}
\nonumber
\mu_t(\mathbb{R}^n)=\tilde{\mu_t}(0)=Z(t,0)=\int_0^\infty \phi(t,\tau)d\tau=1, t>0.
\end{equation}
The solution of the equation (\ref{diffusion_equation}) with the condition (\ref{initial_condition}) has the form:
\begin{equation}
\nonumber
p(t,x)=\int_{\mathbb{R}^n} f(x-y)1(y)\mu_t(dy).
\end{equation}
Thus, if $ f\geq 0 $, then $ p \geq 0 $.
Moreover, if $f$ is bounded, then we have
\begin{equation}
\nonumber
|p(t,x)|\leq \int_{\mathbb{R}^n} |f(x-y)|1(y)\mu_t(dy) \leq \sup_{x\in \mathbb{R}^n} |f(x)| \mu_t(\mathbb{R}^n)=\sup_{x\in \mathbb{R}^n} |f(x)|.
\end{equation}
\end{proof}

\subsection{Special cases}
In this subsection,  some special cases of the general nonlocal diffusion equation are studied.

\subsubsection{Single-term Caputo-Riesz time-space fractional diffusion equation}
We consider the equation (\ref{Caputo_Riesz_diffusion_equation}) with the initial condition (\ref{initial_condition}).
Using the Fourier-Laplace transform to (\ref{Caputo_Riesz_diffusion_equation}), we have
\begin{equation}
\nonumber
s^{\alpha}\hat{\tilde{p}}(s,\xi)-s^{\alpha-1}\tilde{f}(\xi)=-|\xi|^{2\beta}\hat{\tilde{p}}(s,\xi).
\end{equation} 
Then 
\begin{equation}
\nonumber
\hat{\tilde{p}}(s,\xi)=\frac{s^{\alpha-1}\tilde{f}(\xi)}{s^{\alpha}+|\xi|^{2\beta}}.
\end{equation} 
Taking the inverse Laplace transform, we have
\begin{equation}
\nonumber
\tilde{p}(t,\xi)= \tilde f(\xi) E_{\alpha}(-|\xi|^{2\beta} t^\alpha).
\end{equation}
Applying the Fourier transform, we obtain
\begin{equation}
\nonumber
p(t,x)=\frac{1}{(2\pi)^{\frac{n}{2}}} \int_{\mathbb{R}^n} \tilde f(\xi) E_{\alpha}(-|\xi|^{2\beta} t^\alpha) e^{ix\xi} d\xi.
\end{equation}
Since 
\begin{equation}
\nonumber
 E_{\alpha}(-|\xi|^{2\beta} t^\alpha)\leq 1 , t\geq 0,  \xi\in \mathbb{R}^n,
\end{equation}
 we deduce
\begin{equation}
\nonumber
|\tilde{p}(t,\xi)| \leq |\tilde f(\xi)|.
\end{equation}
It follows from $ f\in M_{(k)} $ that $ p(t,\cdot) \in M_{(k)}$, where $ k(x)=c^{-1}|x|^{-(n+2\beta)} $.

\subsubsection{Multi-term Caputo-Riesz time-space fractional diffusion equation}
The equation (\ref{multi_Caputo_Riesz_diffusion_equation}) with the initial condition (\ref{initial_condition}) is studied.
Applying the Fourier transform to (\ref{multi_Caputo_Riesz_diffusion_equation}) and then using Luchko's theorem \cite[Theorem 4.1]{Luchko_Theorem}, we have
\begin{align}
\nonumber
\tilde{p}(t,\xi)&= \tilde f(\xi) \bigg[ 1-\frac{1}{a_1}t^{\alpha_1}\sum_{j=1}^r b_j |\xi|^{2\beta_j}  E_{(\alpha_1-\alpha_2, ...,\alpha_1-\alpha_m,\alpha_1),1+\alpha_1}\\
\nonumber
&\bigg(-\frac{a_2}{a_1} t^{\alpha_1-\alpha_2},...,-\frac{a_m}{a_1} t^{\alpha_1-\alpha_m},-\frac{1}{a_1}t^{\alpha_1}\sum_{j=1}^rb_j|\xi|^{2\beta_j}\bigg)\bigg].
\end{align}
Taking the inverse Fourier transform, we obtain the solution of the form
\begin{align}
\nonumber
p(t,x)=&\frac{1}{(2\pi)^{\frac{n}{2}}} \int_{\mathbb{R}^n} \tilde f(\xi) \bigg[ 1-\frac{1}{a_1}t^{\alpha_1} 
\sum_{j=1}^rb_j |\xi|^{2\beta_j} E_{(\alpha_1-\alpha_2, ...,\alpha_1-\alpha_m,\alpha_1),1+\alpha_1}\\
\nonumber
&\bigg(-\frac{a_2}{a_1} t^{\alpha_1-\alpha_2},...,-\frac{a_m}{a_1} t^{\alpha_1-\alpha_m},-\frac{1}{a_1}t^{\alpha_1}\sum_{j=1}^rb_j |\xi|^{2\beta_j}\bigg)\bigg]e^{ix\xi} d\xi.
\end{align}
It follows from  Lemma 3.2 in \cite{Li_Liu_Yamamoto} that for any $ t>0 $, there exists a $ M_t>0 $ such that for any $ \xi \in \mathbb{R}^n $
\begin{align}
\nonumber
&\frac{1}{a_1}t^{\alpha_1}\sum_{j=1}^rb_j |\xi|^{2\beta_j}E_{(\alpha_1-\alpha_2, ...,\alpha_1-\alpha_m,\alpha_1),1+\alpha_1}\bigg(-\frac{a_2}{a_1} t^{\alpha_1-\alpha_2},...,-\frac{a_m}{a_1} t^{\alpha_1-\alpha_m},\\
\nonumber
&-\frac{1}{a_1}t^{\alpha_1}\sum_{j=1}^rb_j|\xi|^{2\beta_j}\bigg)<M_t.
\end{align}
Then
\begin{equation}
\nonumber
|\tilde{p}(t,\xi)|\leq (1+M_t)|\tilde f(\xi)|.
\end{equation}
It follows from $ f\in M_{(k)} $ that $ p(t,\cdot) \in M_{(k)}$, where $ k(x)=\sum_{j=1}^r \frac{b_j}{c}|x|^{-(n+2\beta_j)}$.

\subsubsection{Tempered fractional diffusion equation  involving fractional Laplacian}
The equation (\ref{tempered_Riesz_diffusion_equation}) with the initial condition (\ref{initial_condition}) is discussed.
Applying the Fourier transform to (\ref{tempered_Riesz_diffusion_equation}), we have
\begin{equation}
\label{fourier_tempered_equation}
\frac{1}{\Gamma(1-\alpha)}\frac{\partial}{\partial t} \int_0^t e^{-b(t-\tau)}(t-\tau)^{-\alpha} \tilde{p}(\tau,\xi)d\tau-e^{-bt} \frac{t^{-\alpha}}{\Gamma(1-\alpha)}\tilde{f}(\xi)=-|\xi|^{2\beta} \tilde{p}(t,\xi).
\end{equation}
Taking the Laplace transform to (\ref{fourier_tempered_equation}), we obtain
\begin{equation}
\label{laplace_fourier_tempered_equation}
 (s+b)^{\alpha-1} (s\hat{\tilde{p}}(s,\xi)-\tilde f(\xi))=-|\xi|^{2\beta} \hat{\tilde{p}}(s,\xi).
\end{equation}
Then
\begin{equation}
\label{laplace_fourier_tempered_equation}
\hat{\tilde{p}}(s,\xi)=\frac{\tilde{f}(\xi)}{s+|\xi|^{2\beta}(s+b)^{1-\alpha}}.
\end{equation}
Moreover, 
\begin{equation}
\label{laplace_fourier_tempered_equation}
\hat{\tilde{p}}(s-b,\xi)=\frac{\tilde{f}(\xi)}{s+|\xi|^{2\beta}s^{1-\alpha}-b}.
\end{equation}
For $ s \in \big\{s\in R: \big | \frac{b}{s+|\xi|^{2\beta}s^{1-\alpha}} \big |<1 \big\} $, we have
\begin{align}
&\frac{1}{s+|\xi|^{2\beta}s^{1-\alpha}-b }=\frac{1}{s+|\xi|^{2\beta}s^{1-\alpha}} \frac{1}{1-\frac{b}{s+\xi^{2\beta}s^{1-\alpha}}} =\frac{1}{s+|\xi|^{2\beta}s^{1-\alpha}}  \nonumber\\
&\sum_{j=0}^{\infty} \frac{b^j}{\big(s+|\xi|^{2\beta}s^{1-\alpha}\big)^j}=\sum_{j=0}^{\infty} \frac{b^j}{\big(s+|\xi|^{2\beta}s^{1-\alpha}\big)^{j+1}}
=\sum_{j=0}^{\infty}  \frac{b^js^{\alpha-(1+j-j\alpha)}}{\big(s^{\alpha}+|\xi|^{2\beta}\big)^{j+1}}.
\end{align}
By formula (1.80) in \cite{Podlubny}:
\begin{equation}
\nonumber
{\Lambda}\big\{t^{\tau j+\gamma-1} E_{\tau,\gamma}^{(j)}(\pm at^\tau) \big\}(s)=\frac{j!s^{\tau-\gamma}}{(s^\tau\mp a)^{j+1}}, s>|a|^{\frac{1}{\tau}},
\end{equation}
where $ E_{\tau,\gamma}^{(j)}(t) $ denotes the $j$th derivative of the two-parameter Mittag-Leffler function
$ E_{\tau,\gamma}(t) $,
we obtain
\begin{equation}
\label{laplace_fourier_tempered_equation}
{\Lambda}^{-1}\bigg\{\frac{1}{s+|\xi|^{2\beta}s^{1-\alpha}-b}\bigg\}(t)=\sum_{j=0}^{\infty} \frac{b^j}{j!}
t^j E_{\alpha,1+j-j\alpha}^{(j)}(-|\xi|^{2\beta}t^{\alpha}).
\end{equation}
By the frequency shifting property of the Laplace transform,  we have
\begin{align}
\nonumber
&e^{bt}\tilde{p}(t,\xi)=\Lambda^{-1}\big\{\hat{\tilde{p}}(s-b,\xi)\big\}(t,\xi)=\tilde{f}(\xi){\Lambda}^{-1}\bigg\{\frac{1}{s+|\xi|^{2\beta}s^{1-\alpha}-b}\bigg\}(t,\xi)\\
\nonumber
&=\tilde{f}(\xi) \sum_{j=0}^{\infty} \frac{b^j}{j!} t^j E_{\alpha,1+j-j\alpha}^{(j)}(-|\xi|^{2\beta}t^{\alpha}).
\end{align}
Using the inverse Fourier transform, we deduce
\begin{align}
\nonumber
&p(t,x)=\frac{1}{(2\pi)^{\frac{n}{2}}}\int_{\mathbb{R}^n} e^{-bt}\tilde{f}(\xi) \sum_{j=0}^{\infty} \frac{b^j}{j!} t^j E_{\alpha,1+j-j\alpha}^{(j)}(-|\xi|^{2\beta}t^{\alpha}) e^{i x \xi} d\xi\\
\nonumber
&= \frac{e^{-bt}}{(2\pi)^{\frac{n}{2}}} \sum_{j=0}^{\infty} \frac{b^j}{j!} t^j \int_{\mathbb{R}^n}\tilde{f}(\xi)  E_{\alpha,1+j-j\alpha}^{(j)}(-|\xi|^{2\beta}t^{\alpha}) e^{i x \xi} d\xi.
\end{align}
The function $  E_{\alpha,1+j-j\alpha}(-t) $ is completely monotone and thus for any $ j=0,1,... $, the function
\begin{equation}
\nonumber
 |E_{\alpha,1+j-j\alpha}^{(j)}(-t)|
\end{equation}
is monotone decreasing.
For any $ j=0,1,... $, we estimate
\begin{equation}
\nonumber
 |E_{\alpha,1+j-j\alpha}^{(j)}(-t)|\leq 1. 
\end{equation}
Then
\begin{equation}
\nonumber
|\tilde{p}(t,\xi)|\leq|\tilde f(\xi)|e^{-bt} \sum_{j=0}^{\infty}  \frac{b^j}{j!} t^j= |\tilde f(\xi)|.
\end{equation}
It follows from $ f\in M_{(k)} $ that $ p(t,\cdot) \in M_{(k)}$, where $ k(x)=c^{-1}|x|^{-(n+2\beta)} $.

\subsubsection{Tempered fractional diffusion equation involving general Laplacian}
 The equation (\ref{tempered_time_space_diffusion_equation}) with the initial condition (\ref{initial_condition}) is investigated.
Using the Fourier transform to (\ref{tempered_time_space_diffusion_equation}), we obtain
\begin{equation}
\label{fourier_tempered_general_equation}
\frac{1}{\Gamma(1-\alpha)}\frac{\partial}{\partial t} \int_0^t e^{-b(t-\tau)}(t-\tau)^{-\alpha} \tilde{p}(\tau,\xi)d\tau-e^{-bt} \frac{t^{-\alpha}}{\Gamma(1-\alpha)}\tilde{f}(\xi)=-\zeta(\xi) \tilde{p}(t,\xi).
\end{equation}
Applying the Laplace transform to (\ref{fourier_tempered_general_equation}), we have
\begin{equation}
\label{laplace_fourier_tempered_equation}
 (s+b)^{\alpha-1} (s\hat{\tilde{p}}(s,\xi)-\tilde f(\xi))=-\zeta(\xi) \hat{\tilde{p}}(s,\xi).
\end{equation}
Then 
\begin{equation}
\label{laplace_fourier_tempered_equation}
\hat{\tilde{p}}(s,\xi)=\frac{\tilde{f}(\xi)}{s+\zeta(\xi)(s+b)^{1-\alpha}}.
\end{equation}
As in Subsubsection 3.3.3, we use the inverse Fourier-Laplace transform to obtain
\begin{equation}
\nonumber
p(t,x)=\frac{e^{-bt}}{(2\pi)^{\frac{n}{2}}} \sum_{j=0}^{\infty} \frac{b^j}{j!} t^j \int_{\mathbb{R}}\tilde{f}(\xi)  E_{\alpha,1+j-j\alpha}^{(j)}(-\zeta(\xi)t^{\alpha}) e^{i x \xi} d\xi.
\end{equation}
By the complete monotonicity of $  E_{\alpha,1+j-j\alpha}(-t) $, we estimate
\begin{equation}
\nonumber
|\tilde{p}(t,\xi)|\leq|\tilde f(\xi)|.
\end{equation}
Then it follows from $ f\in M_{(k)} $ that $ p(t,\cdot) \in M_{(k)}$, where $ k(x)=q|x|^{-(n+2\beta)}e^{-h|x|} $.

\section{IBVP of general fractional diffusion equation}
\label{Section 4}
\setcounter{section}{4}
\setcounter{equation}{0}\setcounter{theorem}{0}
In this section, we consider the following general diffusion equation
\begin{equation}
\label{governing_equation_bounded}
\mathscr{D}_{(g)} p(t,x) =\mathscr{L}_{(k)} p(t,x), t>0, x \in B
\end{equation}
with the initial condition
\begin{equation}
\label{initial_condition_1}
p(0,x)=f(x), x \in B
\end{equation}
and the boundary condition
\begin{equation}
\label{boundary_condition}
p(t,x)=0, t>0, x \in B^C,
\end{equation}
where $ B \subset \mathbb{R}^n $ is an open bounded domain with smooth boundary and $ B^C:=\mathbb{R}^n \setminus B$.
Denote $ H:=\max\limits_{x \in B} |x| $.

In order to obtain the desired result, throughout this section, we assume that the kernel function $k$ satisfies the following condition:\\
(C7) There exist $ \theta>0, 0<\beta< 1 $ such that $ k(x)\geq \theta |x|^{-(n+2\beta)} $ for $ x \in \Omega=[-2H,2H]^n \setminus \{0\} $.\\
It is obvious that the condition (C7) is weaker than the condition (C5). 
We note that the exponentially truncated kernel function $ k $ defined by formula (\ref{typical_general_exam}) satisfies the condition (C7).
With the help of the kernel function $k$, we define the new kernel function $ k^* $ as follows:
\begin{equation}
\nonumber
k^*(x)=\left\{
                \begin{aligned}
                  & k(x) && \text{for $x \in \Omega, $} \\
                  & \theta |x|^{-(n+2\beta)} && \text{for else}.\\
                \end{aligned}
              \right.
\end{equation}
We can easily see that the general  Laplacian $ \mathscr{L}_{( k^*)}$ is a general fractional Laplacian proposed in \cite{Bisci}. 
Now we consider the following equation
\begin{equation}
\label{governing_equation_bounded_changed}
\mathscr{D}_{(g)} p(t,x) =\mathscr{L}_{(k^*)} p(t,x), t>0, x \in B.
\end{equation}
\begin{lemma}
\label{equivalence_lemma}
$p$ is a solution of the equation (\ref{governing_equation_bounded_changed})  with the conditions (\ref{initial_condition_1}) and (\ref{boundary_condition}) if and only if it is a solution of the equation (\ref{governing_equation_bounded})  with the conditions (\ref{initial_condition_1}) and (\ref{boundary_condition}). 
\end{lemma}
\begin{proof}
Suppose that $p$ is a solution of the equation (\ref{governing_equation_bounded_changed})  with the conditions (\ref{initial_condition_1}) and (\ref{boundary_condition}).  If $ x \in B $ and $ |y|>2H $, then $|x+y|> H $ and $ |x-y|>H $. 
Then, for $t>0$ and $x \in B $, we have
\begin{align}
\nonumber
\mathscr{L}_{(k^*)} p(t,x)&=\int_{\mathbb{R}^n} \big(p(t,x+y)+p(t,x-y)-2p(t,x)\big)k^*(y)dy\\
\nonumber
&=\int_{\Omega} \big(p(t,x+y)+p(t,x-y)-2p(t,x)\big)k^*(y)dy\\
\nonumber
&=\int_{\Omega} \big(p(t,x+y)+p(t,x-y)-2p(t,x)\big)k(y)dy\\
\nonumber
&=\int_{\mathbb{R}^n} \big(p(t,x+y)+p(t,x-y)-2p(t,x)\big)k(y)dy=\mathscr{L}_{(k)} p(t,x),
\end{align}
which yields that  $p$  is also a solution of the equation (\ref{governing_equation_bounded})  with the conditions (\ref{initial_condition_1}) and (\ref{boundary_condition}). 

Similarly, we can prove the other direction of the lemma.
\end{proof}

Lemma \ref{equivalence_lemma} shows that the condition (C7) is equivalent to the condition (C5) when solving the IBVP 
(\ref{governing_equation_bounded})-(\ref{boundary_condition}).
From now on, we suppose  that the kernel function $ k $ satisfies the condition (C5).

\subsection{Friedrichs extension of general fractional Laplacian}
In this subsection, based on the results in \cite{Bisci}, the Friedrichs extension of the general fractional Laplacian $ \mathscr{L}_{(k)}$ is discussed. 

For convenience of notation, we denote $ Q:=(\mathbb{R}^n \times \mathbb{R}^n)\setminus (B^C \times B^C) $.
Firstly, we recall concepts and properties of the functional spaces $ X^\beta(B) $ and  $X_0^\beta(B) $,  which are very crucial in our study.
\begin{definition}[\cite{Bisci}]
\label{definition_of_X}
The functional space $ X^\beta(B) $ is the linear space of  Lebesgue measurable functions from $ \mathbb{R}^n $ to $ \mathbb{R} $ such that the restriction to $ B $ of any function $ v \in X^\beta(B)  $ belongs to $ L^2(B) $ and the map
$ (x,y) \rightarrow (v(x)-v(y))\sqrt{k(x-y)} $ is in $ L^2(Q, dxdy) $.
The norm in $X^\beta(B) $ is defined as follows:
\begin{equation}
\nonumber
\|v\|_{X^\beta(B)}=\|v\|_{L^2(B)}+\bigg(\int_Q|v(x)-v(y)|^2 k(x-y)dxdy\bigg)^{1/2}.
\end{equation}
\end{definition}

\begin{definition}[\cite{Bisci}]
The functional space $ X_0^\beta(B) $  is defined by
\begin{equation}
X_0^\beta(B)=\{u\in X^\beta(B):u=0 \text{ a.e.in } B^C \}.
\end{equation}
\end{definition}

\begin{lemma}[Lemma 1.29 in \cite{Bisci}]
\label{Hilbert}
$X_0^\beta(B) $ is a Hilbert space with the scalar product
\begin{equation}
\nonumber
\langle u,v \rangle_{X_0^\beta}=\int_{\mathbb{R}^{2n}} (u(x)-u(y))(v(x)-v(y))k(x-y)dxdy.
\end{equation}
\end{lemma}

\begin{lemma}[Theorem 2.6 in \cite{Bisci}]
\label{density_of_C_0_infinity}
The functional space $ C_0^\infty(B)  $ is dense in $ X_0^\beta(B) $.
\end{lemma}
Now we consider some properties of the operator $ -\mathscr{L}_{(k)} $.

\begin{lemma}[Theorem  1.26 in \cite{Bisci}]
\label{symmetry_of_operator}
The operator $ -\mathscr{L}_{(k)}:C_0^\infty(B) \subset L^2 (B)  \rightarrow L^2 (B) $ is a linear symmetric operator on $L^2 (B)$.
In particular, for $ u, v\in C_0^\infty(B)$, the following equalities hold.
\begin{align}
\nonumber
&\langle -\mathscr{L}_{(k)} u,v \rangle_{L^2 (B)}=\int_{\mathbb{R}^{2n}} \big(u(x+y)+u(x-y)-2u(x)\big)v(x)k(y)dxdy\\
\nonumber
&=\int_{\mathbb{R}^{2n}}(u(x)-u(y))(v(x)-v(y)k(x-y)dxdy=\langle u,v \rangle_{X_0^\beta}\\
\nonumber
&=\int_{\mathbb{R}^{2n}} \big(v(x+y)+v(x-y)-2v(x)\big)u(x)k(y)dxdy=\langle u,-\mathscr{L}_{(k)} v \rangle_{L^2 (B)}.
\end{align}
\end{lemma}

\begin{lemma}[Lemma 1.28 in \cite{Bisci}]
\label{norm_inequality}
There exists a constant $ C>1 $, depending only on $ n, s, \theta $ and $ B $ such that for any $ u \in X_0^\beta(B) $,
\begin{align}
\nonumber
&\int_{Q}|u(x)-u(y)|^2k(x-y)dxdy \leq \|u\|^2_{X^\beta(B)}\\
\nonumber
&\leq C\int_{Q}|u(x)-u(y)|^2k(x-y)dxdy.
\end{align}
\end{lemma}

Therefore, the operator $ -\mathscr{L}_{(k)}:C_0^\infty(B) \subset L^2 (B)  \rightarrow L^2 (B) $ is a strongly monotone operator on $L^2 (B)$.

\begin{theorem}
\label{energetic_space}
The energetic space of the operator $ -\mathscr{L}_{(k)}:C_0^\infty(B) \subset L^2 (B)  \rightarrow L^2 (B) $ is $ X_0^\beta(B) $.
\end{theorem}
\begin{proof}
It follows Lemma \ref{symmetry_of_operator} and Lemma \ref{norm_inequality} that the operator $ -\mathscr{L}_{(k)}:C_0^\infty(B) \subset L^2 (B)  \rightarrow L^2 (B) $ is a linear symmetric, strongly monotone operator on $L^2 (B)$.
Thus, there exists an energetic extension $ -\mathscr{L}_{(k)}^E $ of the operator $ -\mathscr{L}_{(k)} $.
By using Lemma \ref{Hilbert} and Lemma \ref{density_of_C_0_infinity}, we can easily prove the desired result.
\end{proof}

As usual, the Friedrichs extension $ -\mathscr{L}_{(k)}^F $ of the operator $ -\mathscr{L}_{(k)} $ is defined by
\begin{equation}
\nonumber
-\mathscr{L}_{(k)}^F u =-\mathscr{L}_{(k)}^E u \text{ for all } u \in D(-\mathscr{L}_{(k)}^F),
\end{equation}
where $ D(-\mathscr{L}_{(k)}^F)=\{u\in X_0^\beta(B):-\mathscr{L}_{(k)}^E u \in L^2(B)\} $.
This means that $ u \in D(-\mathscr{L}_{(k)}^F)$ iff $ u \in  X_0^\beta(B) $ and there exists a function $ g \in L^2(B) $ such that 
\begin{equation}
\nonumber
\int_{\mathbb{R}^{2n}}(u(x)-u(y))(v(x)-v(y)k(x-y)dxdy=\int_{B}g(x)v(x)dx \text{ for all } v \in X_0^\beta(B).
\end{equation}
The Friedrichs extension $ -\mathscr{L}_{(k)}^F $ of the operator $ -\mathscr{L}_{(k)} $ is self-adjoint, bijective.

\begin{lemma}[Lemma 1.31 in \cite{Bisci}]
\label{compactness_of_X_0}
The embedding $ X_0^\beta(B) \hookrightarrow L^2(B) $ is compact.
\end{lemma}

The eigenvalue problem of the Friedrichs extension $ -\mathscr{L}_{(k)}^F $ is equivalent to the problem that is to find a solution $ u \in X_0^\beta(B), \lambda\in \mathbb{R} $ of the following equation
\begin{equation}
\label{eigenvalue_problem}
\int_{\mathbb{R}^{2n}}(u(x)-u(y))(v(x)-v(y)k(x-y)dxdy=\lambda \int_{B}u(x)v(x)dx \text{ for all } v \in C_0^\infty(B).
\end{equation}
The following results for the eigensolutions of the Friedrichs extension  $ -\mathscr{L}_{(k)}^F $ are standard in references of functional analysis (for example, see \cite{Zeidler}).\\
(P1) The operator $ -\mathscr{L}_{(k)}^F $ has a countable system $\{\psi_j, \lambda_j\}$ of eigensolutions that contain all the eigensolutions of $ -\mathscr{L}_{(k)}^F $.\\
(P2) The eigenvectors $ \{\psi_j\} $ form a complete orthonormal system in $ L^2(B) $. In addition, $ \psi_j \in X_0^\beta(B) $ for all $ j\in N $.\\
(P3) All the eigenvalues $ \lambda_j $ have finite multiplicity. Furthermore,
$ 0<\lambda_1\leq \lambda_2 \leq \cdots $ and $ \lambda_j \rightarrow +\infty $ as $ j \rightarrow \infty $.\\
(P4)
\begin{align}
\nonumber
&D(-\mathscr{L}_{(k)}^F)=\bigg\{u\in L^2(B): \sum_{j=1}^\infty |\lambda_j \langle \psi_j, u\rangle_{L_2(B)}|^2 <\infty \bigg\}.\\
\label{eigensolution_expansion}
&\text{ For any } u \in D(-\mathscr{L}_{(k)}^F), -\mathscr{L}_{(k)}^F u=\sum_{j=1}^\infty \lambda_j \langle \psi_j, u\rangle_{L_2(B)} \psi_j.
\end{align}
In particular, $ D(-\mathscr{L}_{(k)}^F) $ is a Hilbert space with the norm
\begin{equation}
\label{L_K_F_norm}
\|u\|_{D(-\mathscr{L}_{(k)}^F)}=\bigg(\sum_{j=1}^\infty |\lambda_j \langle \psi_j, u\rangle_{L_2(B)}|^2\bigg)^{1/2}.
\end{equation}
For convenience of notation, from now on, $-\mathscr{L}_{(k)}$  will mean the Friedrichs extension $-\mathscr{L}_{(k)}^F$.

\subsection{Existence and properties of solutions}

We use the eigenfunction expansion method to solve the initial boundary value problem.
Let
\begin{equation}
\label{solution_form_1}
p(t,x)=\sum_{j=1}^\infty \omega_j(t) \psi_j(x)
\end{equation}
be the solution of the general diffusion equation (\ref{governing_equation_bounded}) with the conditions
 (\ref{initial_condition_1}) and (\ref{boundary_condition}).
By substituting (\ref{solution_form_1}) into (\ref{governing_equation_bounded}), for any $ j \in N $,
we have 
\begin{align}
\nonumber
&D_{(g)}\omega_j(t)+\lambda_j \omega_j(t)=0, t>0,\\
\nonumber
&\omega_j(0)=\langle \psi_j, f \rangle_{L^2(B)},
\end{align}
which, by Lemma \ref{property_of_general_Caputo}, has a unique solution $ \omega_j(t)$ in $ C[0, \infty) $.
In particular, the solution $ \omega_j(t)$ is a completely monotone function.
\begin{theorem}
Let $ f \in D(-\mathscr{L}_{(k)}) $. Then the general diffusion equation (\ref{governing_equation_bounded}) with the conditions
 (\ref{initial_condition_1}) and (\ref{boundary_condition}) has a unique solution $ p $ in $ C([0, \infty);D(-\mathscr{L}_{(k)})) $.
 In particular, the following relations hold.
 \begin{align}
 \label{p_estimation_L2}
 &\| p(t, \cdot)\|_{L^2(B)}\leq \|f\|_{L^2(B)}, t \geq 0.\\
 \label{p_estimation_DL}
 &\| p(t, \cdot)\|_{D(-\mathscr{L}_{(k)}))}\leq \|f\|_{D(-\mathscr{L}_{(k)})}, t \geq 0.\\
 \label{Dg_p_estimation}
 &\| D_{(g)}p(t, \cdot)\|_{D(-\mathscr{L}_{(k)}))}\leq \|f\|_{D(-\mathscr{L}_{(k)})}, t > 0.\\
\label{p_derivative_estimation_L2}
&\bigg\|\frac{\partial^m p(t,\cdot)}{\partial t^m}\bigg\|_{L^2(B)}\leq \bigg(\frac{m}{et}\bigg)^m\|f\|_{L^2(B)}, m \in \mathbb{N}, t>0 .\\
\label{p_derivative_estimation_DL}
&\bigg\|\frac{\partial^m p(t,\cdot)}{\partial t^m}\bigg\|_{D(-\mathscr{L}_{(k)})}\leq \bigg(\frac{m}{et}\bigg)^m\|f\|_{D(-\mathscr{L}_{(k)})}, m \in \mathbb{N}, t>0 .\\
 \label{p_zero_estimation_DL}
 &\lim_{t \to 0} \|p(t,\cdot)-f\|_{D(-\mathscr{L}_{(k)})}=0.\\
 \label{p_infinity_estimation_DL}
&\lim_{t \to \infty}\|p(t,\cdot)\|_{D(-\mathscr{L}_{(k)})}=0.
 \end{align}
\end{theorem}
\begin{proof}
By Lemma \ref{property_of_general_Caputo}, the function $ |\omega_j(t)|$ is completely monotone on $ (0, \infty) $.  Therefore $ |\omega_j(t)|$ is monotone decreasing on $ [0, \infty) $.
For any $ t \geq 0 $, we have
\begin{align}
\nonumber
&\| p(t, \cdot)\|_{L^2(B)}=\bigg(\sum_{j=1}^\infty | \omega_j(t)|^2\bigg)^{1/2}
\leq \bigg(\sum_{j=1}^\infty |\langle \psi_j, f \rangle_{L^2(B)}|^2\bigg)^{1/2}
=\|f\|_{L^2(B)},\\
\nonumber
&\| p(t, \cdot)\|_{D(-\mathscr{L}_{(k)})}=\bigg(\sum_{j=1}^\infty |\lambda_j \omega_j(t)|^2\bigg)^{1/2}
\leq \bigg(\sum_{j=1}^\infty |\lambda_j\langle \psi_j, f \rangle_{L^2(B)}|^2\bigg)^{1/2}=\|f\|_{D(-\mathscr{L}_{(k)})}.
\end{align}
For $ t>0 $, we estimate
\begin{equation}
\nonumber
\| D_{(g)}p(t, \cdot)\|_{L^2(B)}=\bigg(\sum_{j=1}^\infty | D_{(g)}\omega_j(t)|^2\bigg)^{1/2}
=\bigg(\sum_{j=1}^\infty | \lambda_j\omega_j(t)|^2\bigg)^{1/2}\leq \|f\|_{D(-\mathscr{L}_{(k)})}.
\end{equation}
By (a) of Lemma \ref{property of CMF}, for $ m \in \mathbb{N}, t>0,$ we deduce
\begin{align}
\nonumber
&\bigg\|\frac{\partial^m p(t,\cdot)}{\partial t^m}\bigg\|_{L^2(B)}=\bigg(\sum_{j=1}^\infty\bigg|\frac{\partial^m \omega_j(t)}{\partial t^m}\bigg|^2\bigg)^{1/2}\leq \bigg(\sum_{j=1}^\infty \bigg(\frac{m}{et}\bigg)^{2m}| \omega_j(t)|^2\bigg)^{1/2}\\
\nonumber
&\leq \bigg(\frac{m}{et}\bigg)^m \|f\|_{L^2(B)}.
\end{align}
Similarly, we can easily prove the inequality (\ref{p_derivative_estimation_DL}).
In order to find the asymptotic behavior near $ t=0 $, for $ t>0 $, we have
\begin{align}
\nonumber
&\|p(t,\cdot)-f\|_{D(-\mathscr{L}_{(k)})}=
\bigg(\sum_{j=1}^\infty |\lambda_j (\omega_j(t)-\langle \psi_j, f\rangle_{L^2(B)})|^2\bigg)^{1/2}\\
\nonumber
&\leq \bigg(\sum_{j=1}^\infty |\lambda_j (2\langle \psi_j, f\rangle_{L^2(B)})|^2\bigg)^{1/2}=2\|f\|_{D(-\mathscr{L}_{(k)})}.
\end{align}
Then by Lebesgue's dominated convergence theorem, we estimate
\begin{equation}
\nonumber
\lim_{t \to 0} \|p(t,\cdot)-f\|_{D(-\mathscr{L}_{(k)})}=
\bigg(\sum_{j=1}^\infty \lim_{t \to 0} |\lambda_j (\omega_j(t)-\langle \psi_j, f\rangle_{L^2(B)})|^2\bigg)^{1/2}=0.
\end{equation}
By Young's inequality, we deduce
\begin{align}
\nonumber
&\lambda_j t\omega_j(t) \leq \|\lambda_j \omega_j \|_{L^1(0,t)}=\|\mathscr{D}_{(g)} \omega_j \|_{L^1(0,t)}
\leq \|g\|_{L^1(0,t)}\bigg\|\frac{ d\omega_j(t)}{dt} \bigg\|_{L^1(0,t)}\\
\nonumber
&= \|g\|_{L^1(0,t)}(1-\omega_j (t)).
\end{align}
Then
\begin{equation}
\nonumber
\omega_j (t)\leq \frac{1}{1+\frac{t\lambda_j}{\|g\|_{L^1(0,t)}}}.
\end{equation}
It follows from (b) of Lemma \ref{property of CMF} and Lebesgue's dominated convergence theorem that 
\begin{equation}
\nonumber
\lim_{t \to \infty}\|p(t,\cdot)\|_{D(-\mathscr{L}_{(k)})}=\lim_{t \to \infty}\bigg(\sum_{j=1}^\infty |\lambda_j \omega_j(t)|^2\bigg)^{1/2}=\bigg(\sum_{j=1}^\infty |\lambda_j \lim_{t \to \infty} \omega_j(t)|^2\bigg)^{1/2}=0.
\end{equation}
\end{proof}

\section{Conclusion}
In the paper, we used the CTRW theory to derive the nonlocal diffusion equation with the general Caputo derivative and the general  Laplacian.
The general nonlocal diffusion equation can describe the L\'{e}vy process whose MSD is finite.

By employing the Fourier analysis technique, we investigate the exstence, nonnegativity and boundedness of solutions of the Cauchy problem for the general nonlocal diffusion equation. In particular, for the special cases of the general diffusion equation, the analytical solutions were expressed in terms of the Mittag-Leffler type functions. 

In addition, the existence result for the IBVP of the nonlocal diffusion equation was obtained.

In the future work, we will consider the existence for a more general class of functions and the optimal estimation for the large time behaviour of solutions.

\section*{Acknowledgements}
The authors would like to thank the referees for their valuable advices for the improvement of this paper.

\end{document}